\documentclass[12pt]{article}

\usepackage[margin=0.75in]{geometry}
\usepackage{amsmath, amsthm, amssymb}
\usepackage[mathscr]{euscript}
\usepackage{csquotes}
\usepackage{graphicx}
\usepackage{subfig}
\usepackage{tikz}
\usepackage{hyperref}
\usepackage{cleveref}

\numberwithin{equation}{section}

\newcommand{\prim}{\operatorname{prim}}
\newcommand{\NN}{\mathbb{N}}
\newcommand{\ZZ}{\mathbb{Z}}
\newcommand{\RR}{\mathbb{R}}
\newcommand{\slope}{\operatorname{slope}}

\newcommand{\SL}{\operatorname{SL}}
\newcommand{\FQ}{\mathsf{FQ}}
\newcommand{\slab}{\operatorname{slab}}
\newcommand{\BCZ}{\operatorname{BCZ}}

\newtheorem{proposition}{Proposition}[section]
\newtheorem{theorem}{Theorem}[section]
\theoremstyle{remark}
\newtheorem{remark}{Remark}[section]
\theoremstyle{definition}
\newtheorem{definition}{Definition}[section]

\title{On Cross Sections to the Geodesic and Horocycle Flows on Quotients of $\SL(2, \RR)$ by Hecke Triangle Groups $G_q$}
\author{Diaaeldin Taha \\ University of Washintgon in Seattle, \href{mailto:dtaha@uw.edu}{dtaha@uw.edu}}

\begin{document}
\maketitle

\abstract{In this paper, we provide a model for cross sections to the geodesic and horocycle flows on $\SL(2, \RR)/G_q$ using an extension of a heuristic of P. Arnoux and A. Nogueira. Our starting point is a continued fraction algorithm related to the group $G_q$, and a cross section to the horocycle flow on $\SL(2, \RR)/G_q$ from a previous paper. As an application, we get the natural extension and invariant measure for a symmetric $G_q$-Farey interval map resulting from projectivizing the aforementioned continued fraction algorithm.}

\section{Introduction}

For each integer $q \geq 3$, the \emph{Hecke triangle group $G_q$} is the discrete subgroup of $\SL(2, \RR)$ generated by
\begin{equation}
S := \begin{pmatrix}0 & -1 \\ 1 & 0\end{pmatrix}, \text{ and } T_q := \begin{pmatrix}1 & \lambda_q \\ 0 & 1\end{pmatrix},
\end{equation}
where $\lambda_q := 2 \cos\frac{\pi}{q}$. The modular group $\SL(2, \ZZ)$ is the Hecke triangle group corresponding to $q = 3$. As such, the family of Hecke triangle groups provides a natural test ground for understanding how several classical constructs (e.g. the Farey sequence, Stern-Brocot tree, and so on) can be extended to other Fuchsian groups. One such set of objects that is of interest is the family of the discrete orbits
\begin{equation}
\Lambda_q := G_q (1, 0)^T
\end{equation}
of the linear action of $G_q$ on the plane $\RR^2$. (For $q = 3$, the orbit $\Lambda_3$ is the set of primitive pairs of integers $\ZZ_{\prim}^2 := \{(a, b) \in \ZZ^2 \mid \gcd(a, b) = 1\}$. Those are exactly the non-zero points of $\ZZ^2$ that are visible from the origin.)

In \cite{dia}, we derived an analogue of the classical Stern-Brocot process for the set $\Lambda_q$ of $G_q$-``visible lattice points'' starting with a particular continued fraction algorithm that we refer to, following Janvresse, Rittaud, and De La Rue in \cite{Janvresse2011-fl}, as the \emph{$\lambda_q$-continued fraction algorithm}. Our sought for application for the $G_q$-Stern-Brocot tree in \cite{dia} was deriving an explicit cross section to the horocycle flow on $\SL(2, \RR)/G_q$. In this paper, we study the aforementioned continued fraction algorithm itself, with our focus being the derivation of an explicit cross section to the geodesic flow on $\SL(2, \RR)/G_q$ whose first return map is a natural extension to (a projective version of) the $\lambda_q$-continued fraction algoritm.

Our main tool is an extension of a heuristic introduced by P. Arnoux and A. Nogueira in \cite{Arnoux1993-xh} for deriving geometric models of natural extensions of multidimensional continued fraction algorithms. Our extension of the said heuristic provides a simple picture of the cross section to the horocycle flow $h_\cdot \curvearrowright \SL(2, \RR)/G_q$ we derived in \cite{dia} (\cref{theorem: G_q BCZ maps}), along with the cross section to the geodesic flow $g_\cdot \curvearrowright \SL(2, \RR)/G_q$ we derive here (\cref{proposition: polygon P is a cross section to the horocycle flow}). As a by product, we get an infinite invariant measure for the (projective) $\lambda_q$-continued fraction algorithm. We accelerate the continued fraction algorithm to get a map with a finite invariant measure, and a corresponding finite area geodesic cross section.

In the remainder of this section, we recall the $\lambda_q$-continued fraction algorithm, present our extension of the Arnoux-Nogueira heuristic, and present our main results.

\subsection{The $\lambda_q$-continued fraction algorithm}

As is customary when working with Hecke triangle groups, we write
\begin{equation}
U_q := T_qS = \begin{pmatrix}\lambda_q & -1 \\ 1 & 0\end{pmatrix}.
\end{equation}
For $i = 0, 1, \cdots, q - 2$, we consider the following vectors
\begin{equation}
\mathfrak{w}_i^q = (x_i^q, y_i^q)^T := U_q^i (1, 0)^T
\end{equation}
in $\Lambda_q$. (Note that $\mathfrak{w}_0^q = (1, 0)^T$, $\mathfrak{w}_1^q = (\lambda_q, 1)^T$, $\mathfrak{w}_{q-2}^q = (1, \lambda_q)^T$, and $\mathfrak{w}_{q-1}^q = (0, 1)^T$.) For $i = 0, 1, \cdots, q - 2$, denote by
\begin{equation}
\Sigma_i^q := (0, \infty) \mathfrak{w}_i^q + [0, \infty)\mathfrak{w}_{i+1}^q
\end{equation}
the sector in the first quadrant between $\mathfrak{w}_i^q$ (inclusive), and $\mathfrak{w}_{i+1}^q$ (exclusive).

\begin{figure}
\centering
\begin{tikzpicture}[scale=2]
\draw[->,ultra thin] (-0.5,0)--(2.5,0) node[right]{$\scriptstyle x$};
\draw[->,ultra thin] (0,-0.5)--(0,2.2) node[above]{$\scriptstyle y$};
\fill (1,0) circle[radius=0.5pt];
\fill (1.618,1) circle[radius=0.5pt];
\fill (1.618,1.618) circle[radius=0.5pt];
\fill (1,1.618) circle[radius=0.5pt];
\fill (0,1) circle[radius=0.5pt];
\draw (0, 0) -- (2, 2/1.618);
\draw (0, 0) -- (2, 2);
\draw (0, 0) -- (2/1.618, 2);
\node at (1+0.2,0-0.2) {$\scriptstyle \mathfrak{w}_0^5$};
\node at (1.618+0.3,1-0.1) {$\scriptstyle \mathfrak{w}_1^5$};
\node at (1.618+0.3,1.618) {$\scriptstyle \mathfrak{w}_2^5$};
\node at (1-0.3,1.618+0.1) {$\scriptstyle \mathfrak{w}_3^5$};
\node at (0-0.2,1+0.2) {$\scriptstyle \mathfrak{w}_4^5$};

\node at (2.3, 0.5) {$\scriptstyle \Sigma_0^5$};
\node at (2.3, 1.7) {$\scriptstyle \Sigma_1^5$};
\node at (1.7, 2.3) {$\scriptstyle \Sigma_2^5$};
\node at (0.5, 2.3) {$\scriptstyle \Sigma_3^5$};
\end{tikzpicture}
\caption{The vectors $\left\{\mathfrak{w}_i^5\right\}_{i=0}^4$, and the sectors $\left\{\Sigma_i^5\right\}_{i=0}^3$.}
\end{figure}
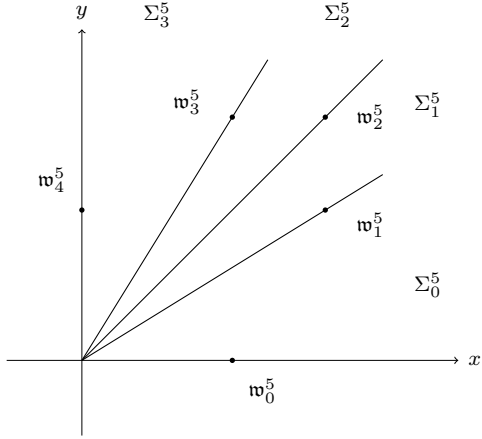

\begin{definition}[\cite{Davis2018-al, Janvresse2011-fl, dia}]
\label{definition: lambda continued fraction}
For any non-zero vector $\mathbf{u} \in \RR^2$, an application of the \emph{$\lambda_q$-continued fraction algorithm} is the following: If $\mathbf{u}$ is in the sector $\Sigma_i^q$ with $0 \leq i \leq q - 2$, then replace $\mathbf{u}$ with $(M_i^q)^{-1} \mathbf{u} \in \cup_{j=0}^{q-2} \Sigma_j^q$, where
\begin{equation}
M_i^q := U_q^{i}T_q = [\mathfrak{w}_i^q\ \mathfrak{w}_{i+1}^q]
\end{equation}
for all $i = 0, 1, \cdots, q - 2$ are the matrices in $G_q$ whose columns are $\mathfrak{w}_i^q$ and $\mathfrak{w}_{i+1}^q$.
\end{definition}

(Note that for every $i = 0, 1, \cdots, q - 2$, the matrix $M_i^q$ bijectively maps $\cup_{j=0}^{q-2} \Sigma_j^q$ to $\Sigma_i^q$.) In \cite{dia}, we showed that the vector $\mathbf{u}$ eventually lands on the line $y = 0$ and is fixed by the algorithm if and only if $\mathbf{u}$ is parallel to a vector in $\Lambda_q$.

\subsection{The Arnoux-Nogueira heuristic}

We below describe a geometric model that can be used to describe cross sections to the horocycle and geodesic flows on $\SL(2, \RR)/G_q$. This model is based on a heuristic introduced by P. Arnoux and A. Nogueira \cite{Arnoux1993-xh} for deriving geometric models of the natural extensions, and invariant measures of multidimensional continued fractions. We extend the heuristic to provide, in our setting, a unified description of cross sections to the horocycle flow in addition to the geodesic flow on $\SL(2, \RR)/G_q$. An excellent presentation of the applications of this heuristic for finding cross sections to the geodesic flow can be found in \cite{Arnoux2018-dg}, or the more classical \cite{Zorich1996-hj}.

Towards that end, we write
\begin{equation}
\FQ := \{(x, y)^T \in \RR^2 \mid x > 0, y \geq 0\}
\end{equation}
for the collection of vectors in the first quadrant, and
\begin{equation}
\widehat{\FQ \times \FQ} := \{(\mathbf{u}, \mathbf{v}) \in \FQ \times \FQ \mid \mathbf{u} \cdot \mathbf{v} = 1\}
\end{equation}
for the collection of pairs of vectors in $\FQ \times \FQ$ with dot product $1$. We have the following elementary linear-algebraic properties for the horocycle and geodesic flows on $\widehat{\FQ \times \FQ}$. (The matrices $h_s, g_{a, b}, g_t$ are as in \cref{equation: h_s}, \cref{equation: gab}, and \cref{equation: g_t}.)

\begin{proposition}
\label{proposition: odds and ends}
Let $\Phi : \widehat{\FQ \times \FQ} \to \SL(2, \RR)$ be the map defined by
\begin{equation}
\Phi\left(\begin{pmatrix}a \\ b\end{pmatrix}, \begin{pmatrix}c \\ d\end{pmatrix}\right) := \begin{pmatrix}a & b \\ -d & c\end{pmatrix}.
\end{equation}
The map $\Phi$ is well-defined, and injective. The following are also true.
\begin{enumerate}
\item For any pair $(\mathbf{u}, \mathbf{v}) \in \widehat{\FQ \times \FQ}$, and matrix $A \in M_{2 \times 2}(\RR)$ with $A^{-1}\mathbf{u}, A^T\mathbf{v} \in \FQ$, we have that $(A^{-1}\mathbf{u}, A^T\mathbf{v}) \in \widehat{\FQ \times \FQ}$, and
\begin{equation}
\Phi(A^{-1}\mathbf{u}, A^T \mathbf{v}) = \Phi(\mathbf{u}, \mathbf{v})(A^{-1})^T.
\end{equation}
\item For any $((a, b)^T, (c, d)^T) \in \widehat{\FQ \times \FQ}$, there exists $s \in [0, \frac{1}{ab})$ such that
\begin{equation*}
\begin{pmatrix}c \\ d\end{pmatrix} = (1 - abs) \begin{pmatrix}1/a \\ 0\end{pmatrix} + abs \begin{pmatrix}0 \\ 1/b\end{pmatrix}.
\end{equation*}
In that case,
\begin{equation}
\Phi((a, b)^T, (c, d)^T) = h_s \Phi((a, b)^T, (1/a, 0)^T) = h_s g_{a, b}.
\end{equation}
\item For any $t \in \RR$, and $(\mathbf{u}, \mathbf{v}) \in \widehat{\FQ \times \FQ}$, we have that
\begin{equation}
g_t\Phi(\mathbf{u}, \mathbf{v}) = \Phi(e^t \mathbf{u}, e^{-t} \mathbf{v}).
\end{equation}
\end{enumerate}
Moreover, the set $\widehat{\FQ \times \FQ}$ can be parametrized by the set
\begin{equation}
\left\{\left((a, b), s\right) \in \FQ \times \RR \mid s \in \left[0, \frac{1}{ab}\right)\right\}.
\end{equation}
\end{proposition}

We apply this heuristic in our setting as follows: The $\lambda_q$-continued fraction from \cref{definition: lambda continued fraction} can be written as a map $\mathbf{A} : \FQ \to \FQ$, where each $\mathbf{u} \in \FQ$ is sent to $A(\mathbf{u})^{-1} \mathbf{u}$, with $A(\mathbf{u}) = M_i^q$ if $\mathbf{u}$ belongs to the sector $\Sigma_i^q$. This extends to a map $\widehat{\mathbf{A}} : \widehat{\FQ \times \FQ} \to \widehat{\FQ \times \FQ}$ that sends each $(\mathbf{u}, \mathbf{v}) \in \widehat{\FQ \times \FQ}$ to $(A(\mathbf{u})^{-1}\mathbf{u}, A(\mathbf{u})^T \mathbf{v})$. In \cref{proposition: polygon P is a cross section to the horocycle flow}, we find a nice fundamental domain in $\widehat{\FQ \times \FQ}$ for the orbits $\widehat{\FQ \times \FQ}/\widehat{\mathbf{A}}$. That is, we find a region in $\widehat{\FQ \times \FQ}$ that contains a unique representative of each $\widehat{\mathbf{A}}$-orbit. Because of the particulars of the $\lambda_q$-continued fraction algorithm, the aforementioned fundamental domain turns out to be identifiable with the quotient $\SL(2, \RR)/G_q$ (up to a set of measure zero). Finally, we describe the return maps of the horocycle and geodesic flows to the boundary of the said domain, giving cross sections to the said flow.

\subsection{Main results}
\label{subsection: main results}

We below overview the main results of this paper. To achieve our goal of presenting a useful, unified description of explicit cross sections to the geodesic and horocycle flows on $\SL(2, \RR)/G_q$, and providing invariant measures for the first return maps to the geodesic cross section, we go through the following three steps.
\begin{enumerate}
\item We start with the suspension $S_{R_q}\mathcal{T}^q$ of the horocycle flow on $\SL(2, \RR)/G_q$ over the $G_q$-Farey triangle $\mathcal{T}^q$ from \cref{theorem: G_q BCZ maps}, and reparametrize it into another suspension $S\mathcal{P}^q$ over what we call the \emph{$G_q$-Stern-Brocot polygon} $\mathcal{P}^q$. The rationale behind this reparametrization is that it provides a simple way to work with the geodesic flow and its coding. It is possible to work directly with $S_{R_q}\mathcal{T}^q$, but that obfuscates the coding, and is mechanically more demanding than going through the reparametrization, and then working with $S\mathcal{P}^q$.
\item Parametrise the sides of $S\mathcal{P}^q$. This gives an explicit cross section to the geodesic flow using a planar region with infinite area. As a by product, we get a natural extension to the $\lambda_q$-continued fraction algorithm, and an infinite invariant measure for a parametrization of the algorithm that we refer to as the \emph{symmetric $G_q$-Farey map $\mathscr{F}_q$}.
\item We accelerate the map $\mathscr{F}_q$ into a another map $\mathscr{G}_q$ that has finite invariant measure. The corresponding natural extension gives a cross section to the geodesic flow using a planar region with finite area.
\end{enumerate}

\subsubsection{The suspension $S\mathcal{P}^q$ of the horocycle flow $h_\cdot \curvearrowright \SL(2, \RR)/G_q$ over the Stern-Brocot polygon $\mathcal{P}^q$}

The proposition stated below is follows from a previous result on a cross section to the horocycle flow on $\SL(2, \RR)/G_q$ (\cref{theorem: G_q BCZ maps}), and the Arnoux-Nogueira heuristic (\cref{proposition: odds and ends}). Motivated by the triangles $\mathcal{T}^q$ from \cref{theorem: G_q BCZ maps} being referred to as \emph{$G_q$-Farey triangles}, we refer to the polygons $\mathcal{P}^q$ in proposition: polygon P is a cross section to the horocycle flow below as \emph{$G_q$-Stern-Brocot polygons}. (In \cite{dia}, the vertices of the Stern-Brocot polygons are involved in the $G_q$-Stern-Brocot process for $\Lambda_q$.)

\begin{figure}
\label{figure: slabs}
\centering
\begin{tikzpicture}[scale=3]
\draw[->,ultra thin] (0,0)--(2,0) node[right]{$\scriptstyle x$};
\draw[->,ultra thin] (0,0)--(0,2) node[above]{$\scriptstyle y$};
\node at (1+0.1,0-0.15) {$\mathfrak{w}_0^5$};
\node at (1.618+0.15,1) {$\mathfrak{w}_1^5$};
\node at (1.618+0.15,1.618+0.15) {$\mathfrak{w}_2^5$};
\node at (1,1.618+0.15) {$\mathfrak{w}_3^5$};
\node at (0-0.15,1+0.15) {$\mathfrak{w}_4^5$};
\draw[thick] (1, 0) -- (1.618, 1);
\draw[thick] (1.618, 1) -- (1.618, 1.618);
\draw[thick] (1.618, 1.618) -- (1, 1.618);
\draw[thick] (1, 1.618) -- (0, 1);
\draw[dashed] (0, 1) -- (1, 0);
\draw[ultra thin] (0, 1) -- (1, 1);
\draw[ultra thin] (1, 0) -- (1, 1.618);
\draw[ultra thin] (1.618, 1) -- (1, 0.382);
\draw[ultra thin] (1.618, 1.618) -- (1, 0.618);
\node at (1.618, 0.5) {$\slab_4^5$};
\node at (1.618+0.25, 1.310+0.05) {$\slab_3^5$};
\node at (1.310+0.06, 1.618+0.25) {$\slab_2^5$};
\node at (0.4, 1.310+0.2) {$\slab_1^5$};
\draw[fill=white] (1,0) circle[radius=0.5pt];
\fill (1.618,1) circle[radius=0.5pt];
\fill (1.618,1.618) circle[radius=0.5pt];
\fill (1,1.618) circle[radius=0.5pt];
\draw[fill=white] (0,1) circle[radius=0.5pt];
\end{tikzpicture}
\caption{The polygon $\mathcal{P}^5$ from \cref{proposition: polygon P is a cross section to the horocycle flow}, along with the slabs $\left\{\slab_i^5\right\}_{i=1}^4$ in the proof of the aforementioned proposition.}
\end{figure}
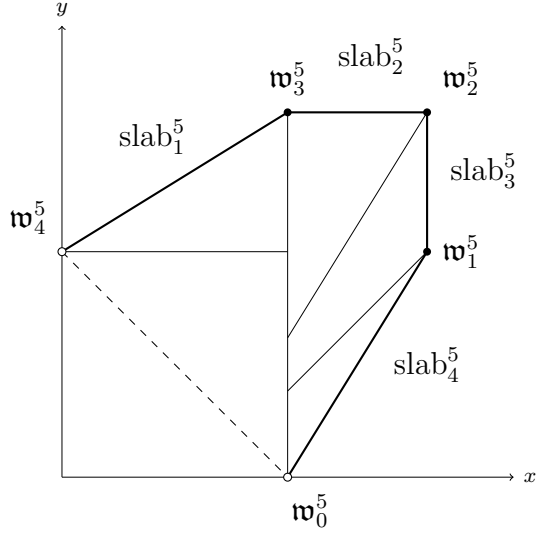

\begin{proposition}
\label{proposition: polygon P is a cross section to the horocycle flow}
Let the polygon $\mathcal{P}^q \subset \FQ$ be the convex hull of the points $\mathfrak{w}_0^q, \cdots, \mathfrak{w}_{q-1}^q$, with the line segment joining $\mathfrak{w}_0^q$ and $\mathfrak{w}_{q - 1}^q$ removed, and denote by
\begin{equation}
S\mathcal{P}^q := \{(\mathbf{u}, \mathbf{v}) \in \widehat{\FQ \times \FQ} \mid \mathbf{u} \in \mathcal{P}^q\}
\end{equation}
the portion of $\widehat{\FQ \times \FQ}$ lying above the polygon $\mathcal{P}^q$. Then the $G_q$-cosets corresponding to the matrices in $\Phi(\widehat{\FQ \times \FQ})$ can be bijectively identified with the points in the suspension $S_{R_q}\mathcal{T}^q$ of the horocycle flow $h_\cdot \curvearrowright \SL(2, \RR)/G_q$ over the Farey triangle $\mathcal{T}^q$. Consequently, the base $\mathcal{P}^q$ of $S\mathcal{P}^q$ is a cross section to the horocycle flow $h_\cdot \curvearrowright \SL(2, \RR)/G_q$, and the open side of $S\mathcal{P}^q$ lying above the line segment joining $\mathfrak{w}_0^q$ and $\mathfrak{w}_{q-1}^q$ is a cross section to the geodesic flow $g_\cdot \curvearrowright \SL(2, \RR)/G_q$.
\end{proposition}

In light of \cref{proposition: odds and ends}, the suspension $S\mathcal{P}^q$ can be parametrised as
\begin{equation*}
S\mathcal{P}^q = \left\{((a, b), s) \in \FQ \times \RR \mid (a, b) \in \mathcal{P}^q, s \in \left[0, \frac{1}{ab}\right)\right\}.
\end{equation*}
An explicit identification of the top ($s = 1/(ab)$) and bottom ($s=0$) of $S\mathcal{P}^q$ that involves the $\BCZ_q$ map from \cref{theorem: G_q BCZ maps} can be derived from the proof of \cref{proposition: polygon P is a cross section to the horocycle flow}. (We have no use for this identification in the current paper, and we omit it.) In \cref{proposition: side identification}, we parametrize the sides of the suspension $S\mathcal{P}^q$ lying above the sides of the polygon $\mathcal{P}^q$. This provids the sought for explicit cross section to the geodesic flow $g_\cdot \curvearrowright \SL(2, \RR)/G_q$. Moreover, tracking the successive closed sides of $S\mathcal{P}^q$ that the geodesic orbit of a point hits gives a discrete coding of the geodesic flow $g_\cdot \curvearrowright \SL(2, \RR)/G_q$ using $q-1$ symbols.

Finally, it should be noted that the side identification of the suspension $S\mathcal{P}^q$ in \cref{proposition: side identification} involves the matrices $\{M_i^q\}_{i=0}^{q-2}$ defining the $\lambda_q$-continued fraction algorithm, hence explaining its relevance to the current work.

\subsubsection{The symmetric Farey and Gauss interval maps for the Hecke triangle group $G_q$ and their natural extensions}

It is immediate that the itineraries of the $\lambda_q$-continued fraction algorithm only depend on the slope of the given vector. We are thus motivated to projectivize the algorithm, and for that we choose the parametrized line segment $\{(a, 1 - a)^T \in \RR^2 \mid a \in (0, 1]\}$ joining the vertices $\mathfrak{w}_0^q = (1, 0)^T$ and $\mathfrak{w}_{q - 1}^q = (0, 1)^T$ of the $G_q$-Stern-Brocot polygon $\mathcal{P}^q$. As we will see in \cref{theorem: natural extension and invariant measure for Farey map}, this choice produces the \emph{symmetric $G_q$-Farey map} $\mathscr{F}_q : (0, 1] \to (0, 1]$ given for any $a \in I_i^q := \left(\frac{1}{1 + \slope(\mathfrak{w}_{i+1}^q)}, \frac{1}{1 + \slope(\mathfrak{w}_i^q)}\right]$, with $i = 0 , 1, \cdots, q - 2$, by
\begin{equation*}
\mathscr{F}_q(a) := \frac{1}{1 + \slope((M_i^q)^{-1} (a, 1 - a)^T)}.
\end{equation*}
In \cref{theorem: natural extension and invariant measure for Farey map}, we show that the side identification of the suspension $S\mathcal{P}^q$ over the $G_q$-Stern-Brocot polygon $\mathcal{P}^q$ is a natural extension of $\mathscr{F}_q$, giving the infinite $\mathscr{F}_q$-invariant measure with density
\begin{equation*}
d\mu_{\mathscr{F}_q} = \frac{da}{a(1-a)}.
\end{equation*}
In \cref{theorem: natural extension and invariant measure for Gauss map}, we accelerate the symmetric $G_q$-Farey map and its natural extension to get the symmetric $G_q$-Gauss map $\mathscr{G}_q$ and a finite $\mathscr{G}_q$-invariant measure $d\mu_{\mathscr{G}_q}$.

\begin{remark}
\label{remark: resolving ambiguity}
As is usually the case with Farey-like maps, there is an ambiguity when it comes to defining the function at the points corresponding to the vectors in $\Lambda_q$. We resolve this ambiguity in what follows by adding the right end points to the sets $\mathsf{S}$, $\mathsf{V}_i^q$, and $\mathsf{H}_i^q$ for all $i \in \{0, 1, \cdots, q - 2\}$ from \cref{proposition: side identification} when we use them in this section. This corresponds to the intervals $\{I_i^q\}_{i = 0}^{q - 2}$ in \cref{theorem: natural extension and invariant measure for Farey map} being closed from the right. We made this particular choice as it gives nice itineraries for the ``ambiguous'' points, and also because this agrees with the choice we made for the continued fraction algorithm in \cite{dia}.
\end{remark}

\section{Preliminaries, the $G_q$-BCZ map, and the suspension $S_{R_q}\mathcal{T}^q$ of the horocycle flow $h_\cdot \curvearrowright \SL(2, \RR)/G_q$}

We below review some notation, properties of vectors and matrices that are related to the groups $G_q$, and an explicit cross section to the horocycle flow $h_\cdot \curvearrowright \SL(2, \RR)/G_q$ from \cite{dia}.

\subsection{Notation}
\label{subsection: notation}

For any integer $n \geq 1$, we at some points write
\begin{equation}
[n] = \{0, 1, \cdots, n - 1\}
\end{equation}
for convenience.

Given two vectors $\mathbf{u}_0 = (x_0, y_0)^T, \mathbf{u}_1 = (x_1, y_1)^T \in \RR^2$, we denote their \emph{(scalar) wedge product} by
\begin{equation}
\mathbf{u}_0 \wedge \mathbf{u}_1 = x_0 y_1 - x_1 y_0,
\end{equation}
and their \emph{dot product} by
\begin{equation}
\mathbf{u}_0 \cdot \mathbf{u}_1 = x_0 x_1 + y_0 y_1.
\end{equation}

Finally, we write
\begin{equation}
\label{equation: h_s}
h_s := \begin{pmatrix}1 & 0 \\ -s & 1\end{pmatrix},
\end{equation}
for $s \in \RR$,
\begin{equation}
\label{equation: g_t}
g_t := \begin{pmatrix}e^t & 0 \\ 0 & e^{-t}\end{pmatrix},
\end{equation}
for $t \in \RR$, and
\begin{equation}
\label{equation: gab}
g_{a, b} = g_{(a, b)^T} := \begin{pmatrix}a & b \\ 0 & a^{-1}\end{pmatrix},
\end{equation}
for $a > 0$, and $b \in \RR$. The above matrices satisfy the identities $g_{e^t, 0} = g_t$, $h_s h_t = h_{s + t}$, and $h_s g_t = g_t h_{se^{2t}}$. Left multplication on $\SL(2, \RR)/G_q$ by $(h_s)_{s \in \RR}$ (resp. $(g_t)_{t \in \RR}$) corresponds to the horocycle flow (resp. geodesic flow) on $\SL(2, \RR)/G_q$.

\subsection{Some properties of the vectors $\{\mathfrak{w}_i^q\}_{i=0}^{q-1}$ and matrices $\{M_i^q\}_{i=0}^{q-2}$}

We have the following elementary properties from \cite{dia} of the vectors $\left\{\mathfrak{w}_i^q\right\}_{i=0}^{q-1}$ and the matrices $\left\{M_i^q\right\}_{i=0}^{q-2}$ that we use throughout the paper.

\begin{proposition}[\cite{dia}]
The following are true.
\begin{enumerate}
\item The vectors $\{\mathfrak{w}_i^q\}_{i=0}^{q-1}$ lie on the ellipse $Q_q(x, y) := x^2 - \lambda_q x y + y^2 = 1$.
\item For any $i = 0, 1, \cdots, q - 2$, we have the Farey neighbor/unimodularity identitiy
\begin{equation}
\mathfrak{w}_i^q \wedge \mathfrak{w}_{i+1}^q = x_i^q y_{i+1}^q - x_{i+1}^q y_i^q = 1,
\end{equation}
along with
\begin{equation}
\mathfrak{w}_0^q \wedge \mathfrak{w}_{q-1}^q = 1.
\end{equation}
\item The set $\Lambda_q = G_q (1, 0)^T$ is symmetric against the line $y = x$, and so
\begin{equation}
(M_i^q)^T = M_{q-2-i}^q
\end{equation}
for all $i = 0, 1, \cdots, q - 2$.
\end{enumerate}
\end{proposition}

\subsection{The $G_q$-BCZ maps, and cross sections to the horocycle flow $h_\cdot \curvearrowright \SL(2, \RR)/G_q$}

We here present our main result from \cite{dia} on a cross section to the horocycle flow $h_\cdot \curvearrowright \SL(2, \RR)/G_q$. The result essentially says the following.

\begin{itemize}
\item The orbit of the ``$G_q$-lattice'' $\Lambda_q$ under the linear action of $\SL(2, \RR)$ on the plane $\RR^2$ can be identified with $\SL(2, \RR)/G_q$. This is an extension of the classical well-known fact that the space of unidmodular lattices can be identified with the cosets $\SL(2, \RR)/\SL(2, \ZZ)$.
\item The $G_q$-Farey triangle $\mathcal{T}^q$, indentified as a subset of $\SL(2, \RR)/G_q$ via the matrices from \cref{equation: gab}, is a cross section to the horocycle flow $h_\cdot \curvearrowright \SL(2, \RR)/G_q$, with $R_q$ as a roof function. (In \cite{dia}, we show that the suspension $S_{R_q}\mathcal{T}^q$ is $\SL(2, \RR)/G_q$ minus particular closed horocycles, which constitute a null set with respect to the Haar measure on $\SL(2, \RR)/G_q$.)
\end{itemize}

In the proof of \cref{proposition: polygon P is a cross section to the horocycle flow}, we show that the suspension $S_{R_q}\mathcal{T}^q$ can be chopped into ``slabs'' and rearranged as the suspension $S\mathcal{P}^q$ in \cref{proposition: polygon P is a cross section to the horocycle flow}, hence revealing an intimate relationship between the definition of the roof function $R_q$ and the coding of the geodesic flow $g_\cdot \curvearrowright \SL(2, \RR)/G_q$. That, and the above two points, are the reasons we need \cref{theorem: G_q BCZ maps} here.

\begin{figure}
\centering
\begin{tikzpicture}[scale=3]
\draw[->,ultra thin] (-0.5,0)--(1.25,0) node[right]{$\scriptstyle a$};
\draw[->,ultra thin] (0,-0.5)--(0,1.25) node[above]{$\scriptstyle b$};
\draw[dashed] (0, 1) -- (1, -0.618);
\draw[thick] (1, -0.618) -- (1, 1) -- (0, 1);
\draw[ultra thin] (0.382, 0.382) -- (1, 0);
\draw[ultra thin] (0.618, 0) -- (1, -0.382);
\node at (1.2, 0.5) {$\scriptstyle \mathcal{T}_4^5$};
\node at (1.2, -0.191) {$\scriptstyle \mathcal{T}_3^5$};
\node at (1.2, -0.5) {$\scriptstyle \mathcal{T}_2^5$};
\end{tikzpicture}

\caption{The $G_5$-Farey triangle $\mathcal{T}^5$ with the subregions $\mathcal{T}_2^5$, $\mathcal{T}_3^5$, and $\mathcal{T}_4^5$ from \cref{theorem: G_q BCZ maps} indicated.}
\end{figure}
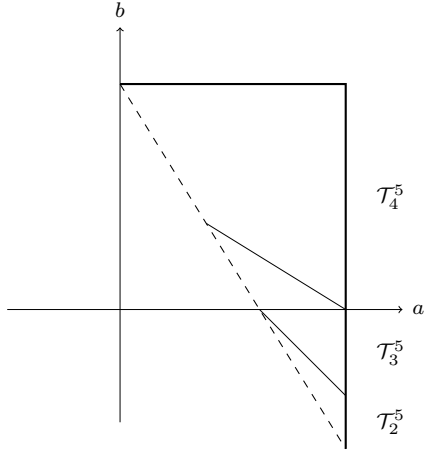

\begin{theorem}[\cite{dia}]
\label{theorem: G_q BCZ maps}
For $\tau > 0$, we write $S_\tau := \{(a, b) \in \RR^2 \mid 0 < a \leq \tau\}$. The following are true.
\begin{enumerate}
\item For any $B \in \SL(2, \RR)$, $B \Lambda_q = \Lambda_q$ if and only if $B \in G_q$. From this follows that the sets $C\Lambda_q$, with $C$ varying over $\SL(2, \RR)$, can be identified with the elements of $\SL(2, \RR)/G_q$.

\item For any $A \in \SL(2, \RR)$, if $A \Lambda_q$ has a horizontal vector of length not exceeding $1$ (i.e. a horizontal vector in $A \Lambda_q \cap S_1$), then $A \Lambda_q$ can be uniquely identified with a point $(a_A, b_A)$ in the \emph{$G_q$-Farey triangle}
\begin{equation}
\mathcal{T}^q = \{(a, b) \in \RR^2 \mid 0 < a \leq 1,\ 1 - \lambda_q a < b \leq 1\}
\end{equation}
through $B\Lambda_q = g_{a_A,b_A}\Lambda_q$. Moreover, the value $a_A$ agrees with the length of the horizontal vector in $A\Lambda_q \cap S_1$.

\item Let $(a, b) \in \mathcal{T}^q$ be any point in the $G_q$-Farey triangle. The set $g_{a, b} \Lambda_q \cap S_1$ has a vector with smallest positive slope. Consequently, there exists a smallest $s = R_q(a, b) > 0$ such that $h_s g_{a,b} \Lambda_q$ has a horizontal vector of length not exceeding $1$, and hence $h_s g_{a, b} \Lambda_q$ corresponds to a unique point $\BCZ_q(a, b) \in \mathcal{T}^q$ in the $G_q$-Farey triangle. The function $R_q : \mathcal{T}^q \to \RR_+$ is referred to as the \emph{$G_q$-roof function}, and the map $\BCZ_q(a, b) : \mathcal{T}^q \to \mathcal{T}^q$ is referred to as the \emph{$G_q$-BCZ map}.

\item The $G_q$-Farey triangle $\mathcal{T}^q$ can be partitioned into the union of
\begin{equation}
\mathcal{T}_i^q := \{(a, b) \in \mathcal{T}^q \mid (a, b)^T \cdot \mathfrak{w}_{i - 1} > 1,\ (a, b)^T\cdot\mathfrak{w}_i \leq 1\},
\end{equation}
with $i = 2, 3, \cdots, q - 1$, such that if $(a, b) \in \mathcal{T}_i^q$, then $g_{a, b}\mathfrak{w}_i^q$ is the vector of least positive slope in $g_{a,b}\Lambda_q \cap S_1$, and
\begin{itemize}
\item the value of the roof function $R_q(a, b)$ is given by
\begin{equation}
R_q(a, b) = R_{q, i}(a, b) := \frac{y_i^q}{a \times (a,b)^T \cdot \mathfrak{w}_i^q}, \text{ and }
\end{equation}
\item the value of the BCZ map $\BCZ_q(a, b)$ is given by
\begin{equation*}
\BCZ_q(a, b) := \left((a,b)^T \cdot \mathfrak{w}_i^q, (a,b)^T \cdot \mathfrak{w}_{i+1}^q + k_i^q(a, b) \times \lambda_q \times (a, b)^T \cdot \mathfrak{w}_i^q\right),
\end{equation*}
where the \emph{$G_q$-index} $k_i^q(a,b)$ is given by
\begin{equation*}
k_i^q(a,b) := \left\lfloor \frac{1 - (a,b)^T \cdot \mathfrak{w}_{i+1}^q}{\lambda_q \times (a,b)^T \cdot \mathfrak{w}_i^q} \right\rfloor.
\end{equation*}
\end{itemize}

\item Let $X_q$ be the homogeneous space $\SL(2, \RR)/G_q$, $\mu_q$ be the probability Haar measure on $X_q$ (i.e. $\mu_q(X_q) = 1$), and $\Omega_q$ be the subset of $X_q$ corresponding to sets $A\Lambda_q$, $A \in \SL(2, \RR)$, with a horizontal vector of length not exceeding $1$. (Note that $\Omega_q$ can be identified with the Farey triangle $\mathcal{T}^q$ via $\left((a, b) \in \mathcal{T}^q\right) \mapsto (g_{a,b}G_q \in \Omega_q)$.) Finally, let $m_q = \frac{2}{\lambda_q}dadb$ be the Lebesgue probability measure on $\mathcal{T}^q$. Then the triple $(\mathcal{T}^q, m_q, \BCZ_q)$, with $\mathcal{T}^q$ identified with $\Omega_q$, is a cross section to $(X_q, \mu_q, h_\cdot)$, with roof function $R_q$.
\end{enumerate}
\end{theorem}

\section{The suspension $S\mathcal{P}^q$ of the horocycle flow $h_\cdot \curvearrowright \SL(2, \RR)/G_q$}

In this section, we prove \cref{proposition: polygon P is a cross section to the horocycle flow}, and give an explicit parametrisation of the side identification of $S\mathcal{P}^q$ in \cref{proposition: side identification}.

\subsection{Proof of \cref{proposition: polygon P is a cross section to the horocycle flow}}

\begin{proof}
Consider the $G_q$-cosets of $h_s g_{a, b}$, with $(a, b)^T$ belonging to the triangle that is the convex hull of the three vectors $\mathfrak{w}_0^q = (1, 0)^T, (1, 1)^T, \mathfrak{w}_{q - 1}^q = (0, 1)^T$ with the line segment between $\mathfrak{w}_0^q$ and $\mathfrak{w}_{q - 1}^q$ removed, and $s \in \left[0, \frac{1}{ab}\right)$. By \cref{proposition: odds and ends}, the cosets in question can be bijectively identified with the portions of the suspensions $S_{R_q}\mathcal{T}^q$ and $S\mathcal{P}^q$ above the aforementioned triangle. (Note that for $S_{R_q}\mathcal{T}^q$, the roof function $R_q$ satisfies $R_q(a, b) = \frac{1}{ab}$ when $(a, b)^T$ belongs to the aforementioned triangle that belongs to $\mathcal{T}_{q-1}^q$.) It thus remains to identify the remainder of the suspensions $S_{R_q}\mathcal{T}^q$ and $S\mathcal{P}^q$.

Let $\mathcal{T}^q{}^\prime$ be the portion of the Farey triangle $\mathcal{T}^q$ outside the triangle from the previous paragraph. We now partition the suspension $S_{R_q}\mathcal{T}^q{}^\prime$ into ``slabs'' as follows: Let
\begin{equation*}
\slab_1^q := \{h_s g_{a, b} \mid (a, b)^T \in \mathcal{T}^q{}^\prime, s \in [R_{q, 0}(a, b), R_{q, 1}(a, b))\},
\end{equation*}
and for $i = 2, \cdots, q - 1$, let
\begin{equation*}
\slab_i^q := \{h_s g_{a, b} \mid (a, b)^T \in \mathcal{T}^q{}^\prime, (a, b)^T \cdot \mathfrak{w}_{i - 1} > 1, s \in [R_{q, i - 1}(a, b), R_{q, i}(a, b)\}.
\end{equation*}
Note that for any $i = 2, 3, \cdots, q - 1$, the collection of points $(a, b)^T \in \mathcal{T}^q{}^\prime$ satisfying $(a, b)^T \cdot \mathfrak{w}_{i - 1} > 1$ is exactly the union $\cup_{j = i}^{q - 1}\mathcal{T}_j^q$, and that for all $(a, b)^T \in \mathcal{T}^q$ with $(a, b)^T \cdot \mathfrak{w}_{i - 1}^q > 1$ (and necessarily $(a, b)^T \cdot \mathfrak{w}_i^q > 0$) that
\begin{eqnarray*}
R_{q, i}(a, b) - R_{q, i - 1}(a, b) &=& \frac{y_i^q}{a (x_i^q a + y_i^q b)} - \frac{y_{i-1}^q}{a(x_{i-1}^q a + y_{i - 1}^q b)} \\
 &=& \frac{1}{a} \left(\frac{(x_{i - 1}^q y_i^q - x_i^q y_{i - 1}^q)a}{(x_i^q a + y_{i - 1}^q b)(x_{i - 1}^q a + y_{i - 1}^q b)}\right) \\
 &=& \frac{1}{((a, b)^T \cdot \mathfrak{w}_i)((a, b)^T \cdot \mathfrak{w}_{i - 1}^q)} \\
 &>& 0,
\end{eqnarray*}
where we used the fact that $\mathfrak{w}_{i-1}^q \wedge \mathfrak{w}_i^q = x_{i-1}^q y_i^q - x_i^q y_{i - 1}^q = 1$. For any $(a, b)^T \in \mathcal{T}^q$, we have $\lambda_q a + b > 1$, and so
\begin{equation*}
R_{q, 1}(a, b) - R_{q, 0}(a, b) = \frac{1}{a(\lambda_q a + b)} - 0 > 0.
\end{equation*}
This implies that the slabs $\slab_1^q, \cdots, \slab_{q-1}^q$ form a partition of $S_{R_q}\mathcal{T}^q{}^\prime$. For any $i = 1, 2, \cdots, q - 1$, if $h_s g_{a, b} \in \slab_i^q$, then by \cref{proposition: odds and ends} we have that
\begin{eqnarray*}
h_s g_{a, b} M_{i-1}^q &=& \Phi\left(\begin{pmatrix}a \\ b\end{pmatrix}, \begin{pmatrix}\frac{1}{a} - bs \\ as\end{pmatrix}\right) M_{i-1}^q \\
 &=& \Phi\left((M_{i-1}^q)^T\begin{pmatrix}a \\ b\end{pmatrix}, (M_{i-1}^q)^{-1}\begin{pmatrix}\frac{1}{a} - bs \\ as\end{pmatrix}\right) \\
 &=& \Phi\left(\begin{pmatrix}(a, b)^T \cdot \mathfrak{w}_{i-1}^q \\ (a, b)^T \cdot \mathfrak{w}_i^q\end{pmatrix}, M_{i-1}^{-1}\begin{pmatrix}\frac{1}{a} - bs \\ as\end{pmatrix}\right).
\end{eqnarray*}
It is evident that $\begin{pmatrix}(a, b)^T \cdot \mathfrak{w}_{i-1}^q \\ (a, b)^T \cdot \mathfrak{w}_i^q\end{pmatrix} \in \FQ$. When $s = R_{q, i - 1}(a, b) = \frac{y_{i-1}^q}{a(a, b)^T \cdot \mathfrak{w}_{i-1}^q}$ we have that
\begin{eqnarray*}
(M_{i-1}^q)^{-1} \begin{pmatrix}\frac{1}{a} - bs \\ as\end{pmatrix} &=& \begin{pmatrix}y_i^q & -x_i^q \\ -y_{i-1}^q & x_{i - 1}^q\end{pmatrix} \begin{pmatrix}\frac{x_{i-1}^q}{(a, b)^T \cdot \mathfrak{w}_{i-1}^q} \\ \frac{y_{i-1}^q}{(a, b)^T \cdot \mathfrak{w}_{i-1}^q}\end{pmatrix} \\
 &=& \begin{pmatrix}\frac{1}{(a, b)^T \cdot \mathfrak{w}_{i-1}^q} \\ 0\end{pmatrix},
\end{eqnarray*}
and when $s = R_{q, i}(a, b) = \frac{y_i^q}{a(a, b)^T \cdot \mathfrak{w}_i^q}$ we have that
\begin{eqnarray*}
(M_{i-1}^q){-1} \begin{pmatrix}\frac{1}{a} - bs \\ as\end{pmatrix} &=& \begin{pmatrix}y_i^q & -x_i^q \\ -y_{i-1}^q & x_{i - 1}^q\end{pmatrix} \begin{pmatrix}\frac{x_i^q}{(a, b)^T \cdot \mathfrak{w}_i^q} \\ \frac{y_i^q}{(a, b)^T \cdot \mathfrak{w}_i^q}\end{pmatrix} \\
 &=& \begin{pmatrix}0 \\ \frac{1}{(a, b)^T \cdot \mathfrak{w}_i^q}\end{pmatrix}, 
\end{eqnarray*}
where we used the fact that $\mathfrak{w}_{i-q}^q \wedge \mathfrak{w}_i^q = x_{i-1}^q y_i^q - x_i^q y_{i - 1}^q = 1$. This implies by \cref{proposition: odds and ends} that as $s$ varies in the interval $[R_{q, i - 1}(a, b), R_{q, i}(a, b))$, then $h_s g_{a, b}$ bijectively identifies with the points $(\mathbf{u}, \mathbf{v}) \in S\mathcal{P}^q$ with $\mathbf{u} = ((a, b)^T \cdot \mathfrak{w}_{i-1}^q, (a, b)^T \cdot \mathfrak{w}_i^q)^T$. It thus remains to show that the bases of the slabs $\slab_1^q, \cdots, \slab_{q-1}^a$ are mapped by $(M_0^q)^T, \cdots, (M_{q-2}^q)^T$ into a partition of the remaining part of the polygon $\mathcal{P}^q$.

The base of the slab $\slab_1^q$ is a triangle with vertices $(1, 1 - \lambda_q)^T, (1, 0)^T, (0, 1)^T$ (with the line segment between $(1, 1 - \lambda_q)$ and $(0, 1)^T$ removed), and the vertices in question are mapped by $(M_0^q)^T = \begin{pmatrix}1 & 0 \\ \lambda_q & 1\end{pmatrix}$ to $(1, 1)^T, (1, \lambda_q)^T = \mathfrak{w}_{q-2}^q, (0, 1)^T = \mathfrak{w}_{q-1}^q$. For a fixed $i = 2, 3, \cdots, q - 1$, the base of the slab $\slab_i^q$ is a quadrilateral\footnote{It should be noted that for $\slab_2^q$, the points $B_2^q$ and $(0, 1)^T$ agree, and for $\slab_{q-1}^q$, the points $A_{q-1}^q$ and $(1, 0)^T$ agree. We also go through our computations with the understanding that $0/0=1$.} with vertices $A_i^q = \left(1, \frac{1 - x_{i-1}^q}{y_{i-1}^q}\right), (1, 0)^T, (0, 1)^T, B_i^q = \left(\frac{y_{i-1}^q - 1}{\lambda_q y_{i-1}^q - x_{i-1}^q}, \frac{\lambda_q - x_{i-1}^q}{\lambda_q y_{i-1}^q - x_{i-1}^q}\right)$ (with both the line segment between $(0, 1)^T$ and $B_i^q$, and the line segment between $B_i^q$ and $A_i^q$ removed). (The point $A_i^q$ is the solution of the two equations $a = 1$ and $(a, b)^T \cdot \mathfrak{w}_{i-1}^q = 1$, and $B_i^q$ is the solution of $\lambda_q a + b = 1$ and $(a, b)^T \cdot \mathfrak{w}_{i-1}^q = 1$.) The matrix $(M_{i-1}^q)^T = M_{q - 1 - i}^q$ maps the point $(1, 0)^T$ to $\mathfrak{w}_{q - 1 - i}^T$, the point $(0, 1)^T$ to $\mathfrak{w}_{q - i}^T$, the point $A_i^q$ to
\begin{eqnarray*}
(M_{i-1}^q)^T A_i^q &=& \begin{pmatrix}x_{i-1}^q & y_{i-1}^q \\ x_i^q & y_i^q\end{pmatrix} \begin{pmatrix}1 \\ \frac{1 - x_{i-1}^q}{y_{i-1}^q}\end{pmatrix} \\
 &=& \begin{pmatrix}1 \\ \frac{y_i^q - 1}{y_{i-1}^q}\end{pmatrix},
\end{eqnarray*}
and the point $B_i^q$ to
\begin{eqnarray*}
(M_{i-1}^q)^T B_i^T &=& \begin{pmatrix}x_{i-1}^q & y_{i-1}^q \\ x_i^q & y_i^q\end{pmatrix} \begin{pmatrix}\frac{y_{i-1}^q - 1}{\lambda_q y_{i-1}^q - x_{i - 1}^q} \\ \frac{\lambda_q - x_{i-1}^q}{\lambda_q y_{i-1}^q - x_{i-1}^q}\end{pmatrix} \\
 &=& \begin{pmatrix}1 \\ \frac{\lambda_q y_i^q - x_i^q + x_i^q y_{i-1}^q - x_{i-1}^q y_i^q}{\lambda_q y_{i-1}^q - x_{i-1}^q}\end{pmatrix} \\
 &=& \begin{pmatrix}1 \\ \frac{y_{i-1}^q - 1}{y_{i-2}^q}\end{pmatrix}
\end{eqnarray*}
where we used the facts that $\mathfrak{w}_{i-1}^q \wedge \mathfrak{w}_i^q = x_{i-1}^q y_i^q - x_i^q y_{i-1}^q = 1$, $\begin{pmatrix}x_{i-1}^q \\ y_{i-1}^q\end{pmatrix} = \begin{pmatrix}\lambda_q & -1 \\ 1 & 0\end{pmatrix}^{-1} \begin{pmatrix}x_i^q \\ y_i^q\end{pmatrix}$, and $\begin{pmatrix}x_{i-2}^q \\ y_{i-2}^q\end{pmatrix} = \begin{pmatrix}\lambda_q & -1 \\ 1 & 0\end{pmatrix}^{-1} \begin{pmatrix}x_{i-1}^q \\ y_{i-1}^q\end{pmatrix}$. For $i = 3, \cdots, q - 1$, we have that $M_{i-1}^T B_i^q = M_{i-2}^T A_{i-1}^q$ is a point on the line $a = 1$. Moreover, $(M_{q-2}^q)^T A_{q-1}^T = (1, 0)^T = \mathfrak{w}_1^q$, and $(M_1^q)^T B_2^q = (1, \lambda_q)^T = \mathfrak{w}_{q-2}^q$. This proves that bases of $\{\slab_i^q (M_{i-1}^q)^T\}_{i=1}^{q-1}$ partition $\mathcal{P}^q$ as required.

That $\mathcal{P}^q$ forms a cross section to the horocycle flow follows from \cref{proposition: odds and ends} and the fact that $\mathcal{T}^q$ is such a cross section.

That the sides of $S\mathcal{P}^q$ follows from \cref{proposition: odds and ends}, and the fact that any ray in $\FQ$ passing through the origin intersects the sides of $\mathcal{P}^q$. This is expanded on in \cref{proposition: side identification} and its proof.
\end{proof}

\subsection{Side identification of the suspension $S\mathcal{P}^q$}

\begin{proposition}
\label{proposition: side identification}
Consider the set
\begin{equation}
\mathsf{S} := \left\{(a, s) \in \RR^2 \mid a \in (0, 1), s \in \left[0, \frac{1}{a(1-a)}\right)\right\} \bigcup \left\{(1, s) \in \RR^2 \mid s \in [0, \lambda_q)]\right\}
\end{equation}
endowed with the Lebesgue measure $d\mu_\mathsf{S} = dads$, and understood, in light of \cref{proposition: odds and ends}, to be parametrizing the portion of $\widehat{\FQ \times \FQ}$ lying above the line segment joining the vectors $\mathfrak{w}_0^q$ and $\mathfrak{w}_{q-1}^q$ via $(a, s) \mapsto h_s g_{a, 1 - a}$. Similarly, for $i = 0, 1, \cdots, q - 2$, consider the set
\begin{equation}
\mathsf{S}_i^q :=
\begin{cases}
\left\{(\alpha, \sigma) \in \RR^2 \mid \alpha \in (0, 1), s \in \left[0, \frac{1}{(\alpha x_i^q + (1-\alpha)x_{i+1}^q)(\alpha y_i^q + (1 - \alpha)y_{i+1}^q)}\right)\right\}, & i = 0 \\
\left\{(\alpha, \sigma) \in \RR^2 \mid \alpha \in (0, 1], s \in \left[0, \frac{1}{(\alpha x_i^q + (1-\alpha)x_{i+1}^q)(\alpha y_i^q + (1 - \alpha)y_{i+1}^q)}\right)\right\}, & i \neq 0
\end{cases}
\end{equation}
endowed with the Lebesgue measure $d\mu_{\mathsf{S}_i^q} = d\alpha d\sigma$, and understood to be parametrising the portion of $\widehat{\FQ \times \FQ}$ lying above the line segment joining the vectors $\mathfrak{w}_i^q$ and $\mathfrak{w}_{i+1}^q$ via $(\alpha, \sigma) \mapsto h_\sigma g_{\alpha \mathfrak{w}_i^q + (1 - \alpha)\mathfrak{w}_{i+1}^q}$. Futhermore, consider the partitioning of $\mathsf{S}$ into $q - 1$ horizontal strips
\begin{equation}
\mathsf{H}_i^q := \left\{(a, s) \in \mathsf{S} \mid s \in [R_{q,i}(a, 1-a), R_{q, i+1}(a, 1-a))\right\}
\end{equation}
with $i \in [q - 1]$, and $q - 1$ vertical strips
\begin{equation}
\mathsf{V}_i^q :=
\begin{cases}
\left\{(a, s) \in \mathsf{S} \mid \slope((a, 1-a)^T) \in \left(\slope(\mathfrak{w}_i^q), \slope(\mathfrak{w}_{i+1}^q)\right)\right\}, & i = 0 \\
\left\{(a, s) \in \mathsf{S} \mid \slope((a, 1-a)^T) \in \left[\slope(\mathfrak{w}_i^q), \slope(\mathfrak{w}_{i+1}^q)\right)\right\}, & i \neq 0
\end{cases}
\end{equation}
with $i \in [q - 1]$. The following are true.
\begin{enumerate}
\item For any $i \in [q - 1]$, and for any $(a, s) \in \mathsf{V}_i^q$, the smallest $t > 0$ such that --- in light of \cref{proposition: odds and ends} and the above parametrizations --- $g_t h_s g_{a, 1 - a}$ is in $\mathsf{S}_i^q$ is $t = -\log\rho_i^q(a)$, where
\begin{equation}
\rho_i^q(a) = (x_{i+1}^q - y_i^q) a + (x_i^q - x_{i+1}^q).
\end{equation}
The map $\mathsf{V}_i^q \to \mathsf{S}_i^q$ induced by
\begin{equation}
h_s g_{a, 1 - a} \mapsto h_\sigma g_{\alpha \mathfrak{w}_i^q + (1 - \alpha) \mathfrak{w}_{i+1}^q} = g_{-\log \rho_i^q(a)} h_s g_{a, 1 - a}
\end{equation}
is bijective, and its Jacobian determinant is equal to $1$.
\item For any $i \in [q - 1]$, and for any $(\alpha, \sigma) \in \mathsf{S}_i^q$, the map $\mathsf{S}_i^q \to \mathsf{H}_{q - 2 - i}^q$ induced by
\begin{equation}
h_\sigma g_{\alpha \mathfrak{w}_i^q + (1 - \alpha)\mathfrak{w}_{i+1}^q} \mapsto h_s g_{a, 1 - a} = h_\sigma g_{\alpha \mathfrak{w}_i^q + (1 - \alpha)\mathfrak{w}_{i+1}^q} ((M_i^q)^{-1})^T
\end{equation}
is bijective, and its Jacobian determinant is equal to $1$.
\end{enumerate}
\end{proposition}

\begin{proof}
We begin by proving the first numbered claim. The bijectivity of the map $\mathsf{V}_i^q \to \mathsf{S}_i^q$ in question follows from \cref{proposition: odds and ends}, and the fact that a point $((a, b)^T, (1/a, 0)^T)$ on the ``floor'' of $\widehat{\FQ \times \FQ}$ is mapped by the geodesic flow to another point $((e^t a, e^t b)^T, (1/(e^t a), 0)^T)$ on the floor, and a point $((a, b)^T, (0, 1/b)^T)$ on the ``roof'' of $\widehat{\FQ \times \FQ}$ is mapped by the geodesic flow to another point $((e^t a, e^t b)^T, (0, 1/(e^t b))^T)$ on the roof. It thus remains to compute the hit times and the Jacobian determinant of the map $\mathsf{V}_i^q \to \mathsf{S}_i^q$. The equation of the straight line $L_i^q$ going through the points $\mathfrak{w}_i^q$ and $\mathfrak{w}_{i+1}^q$ is $(y_{i+1}^q - y_i^q) a + (-x_{i+1}^q + x_i^q) b = 1$. The line $L_i^q$ and the line going through the origin $(0, 0)^T$ and $(a, 1 - a)^T$ intersect at the point
\begin{equation*}
\left(\frac{a}{(y_{i+1}^q - y_i^q) a + (-x_{i+1}^q + x_i^q)(1-a)}, \frac{1-a}{(y_{i+1}^q - y_i^q) a + (-x_{i+1}^q + x_i^q)(1 - a)}\right)^T = \left(\frac{a}{\rho_i^q(a)}, \frac{1-a}{\rho_i^q(a)}\right)^T,
\end{equation*}
and the hitting time is $t = \log\frac{1}{\rho_i^q(a)}$ as claimed. We have that $\begin{pmatrix}x_{i+1}^q \\ y_{i+1}^q\end{pmatrix} = \begin{pmatrix}\lambda_q & -1 \\ 1 & 0\end{pmatrix} \begin{pmatrix}x_i^q \\ y_i^q\end{pmatrix}$, and so we get the expression for $\rho_i^q(a)$ in the statement. We thus have
\begin{eqnarray*}
h_\sigma g_{\alpha \mathfrak{w}_i^q + (1 - \alpha)\mathfrak{w}_{i+1}^q} &=& g_{-\log\rho_i^q(a)} h_s g_{a, 1 - a} \\
\begin{pmatrix}1 & 0 \\ -\sigma & 1\end{pmatrix} \begin{pmatrix}\alpha x_i^q + (1 - \alpha)x_{i+1}^q & \alpha y_i^q + (1 - \alpha)y_{i+1}^q \\ 0 & \ast \end{pmatrix} &=& \begin{pmatrix}\frac{1}{\rho_i^q(a)} & 0 \\ 0 & \rho_i^q(a)\end{pmatrix} \begin{pmatrix}1 & 0 \\ -s & 1\end{pmatrix} \begin{pmatrix}a & 1 - a \\ 0 & \ast\end{pmatrix} \\
\begin{pmatrix}\alpha x_i^q + (1 - \alpha)x_{i+1}^q & \alpha y_i^q + (1 - \alpha)y_{i+1}^q \\ -\sigma(\alpha x_i^q + (1 - \alpha)x_{i+1}^q) & \ast \end{pmatrix} &=& \begin{pmatrix}\frac{a}{\rho_i^q(a)} & \frac{1 - a}{\rho_i^q(a)} \\ -as\rho_i^q(a) & \ast\end{pmatrix}
\end{eqnarray*}
If $x_i^q - x_{i+1}^q \neq 0$, we use the identity $\alpha x_i^q + (1 - \alpha) x_{i+1}^q = \frac{a}{\rho_i^q(a)}$ to compute $\frac{d\alpha}{da}$ and $\frac{d\alpha}{ds}$. If $x_i^q - x_{i+1}^q = 0$, then $y_i^q - y_{i+1}^q \neq 0$, and we use the identity $\alpha y_i^q + (1 - \alpha)y_{i+1}^q = \frac{1 - a}{\rho_i^q(a)}$ instead to find the aforementioned derivatives. We proceed assuming that $x_i^q - x_{i+1}^q \neq 0$, with the other case being similar. We thus have that
\begin{equation*}
\frac{d\alpha}{da} = \frac{1}{x_i^q - x_{i+1}^q} \frac{\rho_i^q(a) - (\rho_i^q)^\prime(a)a}{\rho_i^q(a)^2} = \frac{1}{\rho_i^q(a)^2},
\end{equation*}
and
\begin{equation*}
\frac{d\alpha}{ds} = 0,
\end{equation*}
and so we have that $\frac{\partial(\alpha, \sigma)}{\partial(a, s)} = \left|\frac{d\alpha}{da}\frac{d\sigma}{ds}\right|$. We also have that
\begin{equation*}
\frac{d\sigma}{ds} = \frac{a\rho_i^q(a)}{\alpha x_i^q + (1 - \alpha)x_{i+1}^q} = \rho_i^q(a)^2.
\end{equation*}
This proves that $\frac{\partial(\alpha, \sigma)}{\partial(a, s)} = 1$ as required.

We now prove the second numbered claim. By \cref{proposition: odds and ends}, and the proof of \cref{proposition: polygon P is a cross section to the horocycle flow}, we get that the map $\mathsf{S}_i^q \to \mathsf{H}_{q - 2 - i}^q$ in question is bijective. The reason is that $\mathsf{S}_i^q$ parametrises the side of $\slab_{q - 1 - i}^q M_{q - 2 - i}^q$ lying above the line segment joining $\mathfrak{w}_i^q$ and $\mathfrak{w}_{i+1}^q$, and $\mathsf{H}_{q-2-i}^q$ parametrises the side of $\slab_{q - 1 - i}^q$ lying above the line segment joining $(1, 0)^T$ and $(0, 1)^T$, and right multiplication by $(M_{q-2-i}^q)^{-1} = ((M_i^q)^{-1})^T$ maps $\mathsf{S}_i^q$ bijectively back to $\mathsf{H}_{q - 2 - i}^q$. It remains to show that its Jacobian determinant is $1$. We have that
\begin{eqnarray*}
h_s g_{a, 1 - a} &=& h_\sigma g_{\alpha \mathfrak{w}_i^q + (1 - \alpha)\mathfrak{w}_{i+1}^q} ((M_i^q)^{-1})^T \\
\begin{pmatrix}a & \ast \\ -as & \ast\end{pmatrix} &=& h_\sigma \begin{pmatrix}\alpha x_i^q + (1 - \alpha)x_{i+1}^q & \alpha y_i^q + (1-\alpha)y_{i+1}^q \\ 0 & \ast\end{pmatrix} \begin{pmatrix}y_{i+1}^q & -y_i^q \\ -x_{i+1}^q & x_i^q\end{pmatrix} \\
 &=& \begin{pmatrix}1 & 0 \\ -\sigma & 1\end{pmatrix} \begin{pmatrix}\alpha & \ast \\ -\frac{x_{i+1}^q}{\alpha x_i^q + (1=\alpha)x_{i+1}^q} & \ast\end{pmatrix} \\
 &=& \begin{pmatrix}\alpha & \ast \\ -\frac{x_{i+1}^q}{\alpha x_i^q + (1=\alpha)x_{i+1}^q} - \sigma \alpha & \ast \end{pmatrix}
\end{eqnarray*}
where we used the fact that $\mathfrak{w}_i^q \wedge \mathfrak{w}_{i+1}^q = x_i^q y_{i+1}^q - x_{i+1}^q y_i^q = 1$. We thus have that $\frac{da}{d\alpha} = 1$, $\frac{da}{d\sigma} = 0$, and $\frac{ds}{d\sigma} = 1$, and so $\frac{\partial(a, s)}{\partial(\alpha, \sigma)} = 1$ as required.
\end{proof}

\section{The symmetric Farey and Gauss interval maps for the Hecke triangle group $G_q$, and their natural extensions}

In this section, we use the identification of the sides of $S\mathcal{P}^q$ from \cref{proposition: side identification} to produce in \cref{theorem: natural extension and invariant measure for Farey map} a model of the natural extension of the Farey map corresponding to our choice of projectivization for the $\lambda_q$-continued fraction algorithm, along with an infinite invariant measure for the Farey map. We also accelerate the Farey map and its natural extension, to get a Gauss map with finite invariant measure in \cref{theorem: natural extension and invariant measure for Gauss map}.

\subsection{The symmetric $G_q$ Farey interval map $\mathscr{F}_q$, and its natural extension}

\begin{figure}
\centering
\includegraphics[scale=0.75]{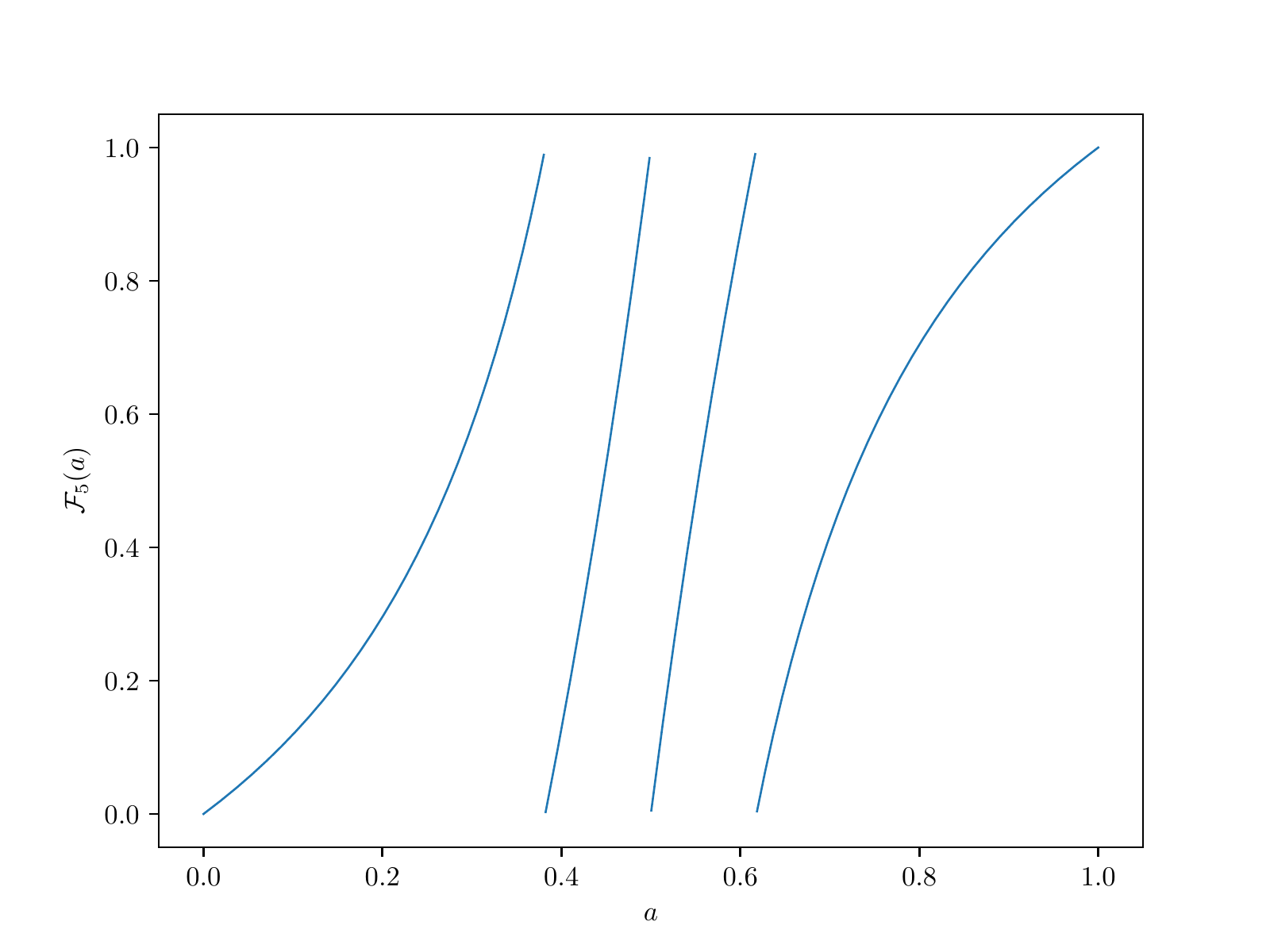}
\caption{The symmetric $G_5$-Farey interval map $\mathscr{F}_5$.}
\end{figure}

\begin{theorem}
\label{theorem: natural extension and invariant measure for Farey map}
Using the notation in \cref{proposition: side identification} and the arbitrary choice we make in \cref{remark: resolving ambiguity}, the map $\widetilde{\mathscr{F}}_q : \mathsf{S} \to \mathsf{S}$ that for any $i \in [q - 1]$ sends any point $(a, s) \in \mathsf{V}_i^q \subset \mathsf{S}$ to the point $(a^\prime, s^\prime)$ given, according to \cref{proposition: odds and ends}, by
\begin{equation}
h_{s^\prime} g_{a^\prime, 1 - a^\prime} = g_{-\log \rho_i^q(a)} h_s g_{a, 1 - a} ((M_i^q)^{-1})^T
\end{equation}
preserves the Lebesgue measure $d\mu_{\widetilde{F}_q} := d\mu_\mathsf{S} = dads$ on $\mathsf{S}$. Moreever, the map $\widetilde{\mathscr{F}}_q$ satisfies the Markov condition $\widetilde{\mathscr{F}}_q(\mathsf{V}_i^q) = \mathsf{H}_{q - 2 - i}^q$ for all $i \in [q - 1]$.

Consequently, $\widetilde{\mathscr{F}}_q$ is a model of the natural extension of the map $\mathscr{F}_q : (0, 1] \to (0, 1]$ defined for all $i \in [q - 1]$ and
\begin{eqnarray}
a \in I_i^q &:=& \left(\frac{1}{1 + \slope(\mathfrak{w}_{i+1}^q)}, \frac{1}{1 + \slope(\mathfrak{w}_i^q)}\right] \\
&=& \left(\frac{x_{i+1}^q}{x_{i+1}^q + y_{i+1}^q}, \frac{x_i^q}{x_i^q + y_i^q}\right]
\end{eqnarray}
by
\begin{eqnarray}
\mathscr{F}_q(a) &:=& \frac{1}{1 + \slope\left((M_i^q)^{-1} (a, 1 - a)^T\right)} \\
 &=& \frac{(x_{i+1}^q + y_{i+1}^q)a - x_{i+1}^q}{(x_{i+1}^q - y_i^q)a + (x_i^q - x_{i+1}^q)}.
\end{eqnarray}
Moreover, the map $\mathscr{F}_q$ preserves the measure
\begin{equation}
d\mu_{\mathscr{F}_q} := \frac{da}{a(1 - a)}
\end{equation}
on $(0, 1]$.
\end{theorem}

\begin{proof}
The map $\widetilde{\mathscr{F}}_q : \mathsf{S} \to \mathsf{S}$ is the composition of the maps from the two numbered claims in \cref{proposition: side identification}. This proves that $\widetilde{\mathscr{F}}_q$ preserves $d\mu_{\widetilde{F}_q}$ and satisfies the sought for Markov condition.

For any $i \in [q - 1]$, if $(a, \ast)$ belongs to $\mathsf{V}_i^q$, then $(a, 1 - a)^T$ belongs to the sector $\Sigma_i^q$, and so the slope of $(a, 1 - a)^T$  is in $[\slope(\mathfrak{w}_i^q), \mathfrak{w}_{i+1}^q)$. This corresponds to $a \in I_i^q$ in the claim of the theorem. By \cref{proposition: odds and ends}, the image $(a^\prime, \ast)$ of $(a, \ast)$ under $\widetilde{\mathscr{F}}_q$ is given by
\begin{eqnarray*}
\Phi((a^\prime, 1 - a^\prime)^T, \ast) &=& g_{-\log\rho_i^q(a)} \left(\Phi((a, 1 - a)^T, \ast)\right) ((M_i^q)^{-1})^T \\
 &=& \Phi\left(\frac{1}{\rho_i^q(a)} (M_i^q)^{-1} (a, 1 - a)^T, \ast\right),
\end{eqnarray*}
and so
\begin{eqnarray*}
\frac{1 - a^\prime}{a^\prime} &=& \slope\left(\frac{1}{\rho_i^q(a)} (M_i^q)^{-1} (a, 1 - a)^T\right) \\
 &=& \slope\left(\begin{pmatrix}y_{i+1}^q & -x_{i+1}^q \\ -y_i^q & x_i^q\end{pmatrix} \begin{pmatrix}a \\ 1 - a\end{pmatrix}\right) \\
 &=& \frac{-(x_i^q + y_i^q)a + x_i^q}{(x_{i+1}^q + y_{i+1}^q)a - x_{i+1}^q}.
\end{eqnarray*}
This, along with the fact that $\begin{pmatrix}x_{i+1}^q \\ y_{i+1}^q\end{pmatrix} = \begin{pmatrix}\lambda_q & -1 \\ 1 & 0\end{pmatrix} \begin{pmatrix}x_i^q \\ y_i^q\end{pmatrix}$ give the expression for $\mathscr{F}_q(a)$ in the theorem. Integrating the measure $d\mu_{\widetilde{\mathscr{F}}_q}$ along the $s$ fibers gives $d\mu_{\mathscr{F}_q}$. This proves the claim.
\end{proof}

\subsection{The symmetric $G_q$ Gauss map $\mathscr{G}_q$, and its natural extension}

\begin{figure}
\centering
\includegraphics[scale=0.75]{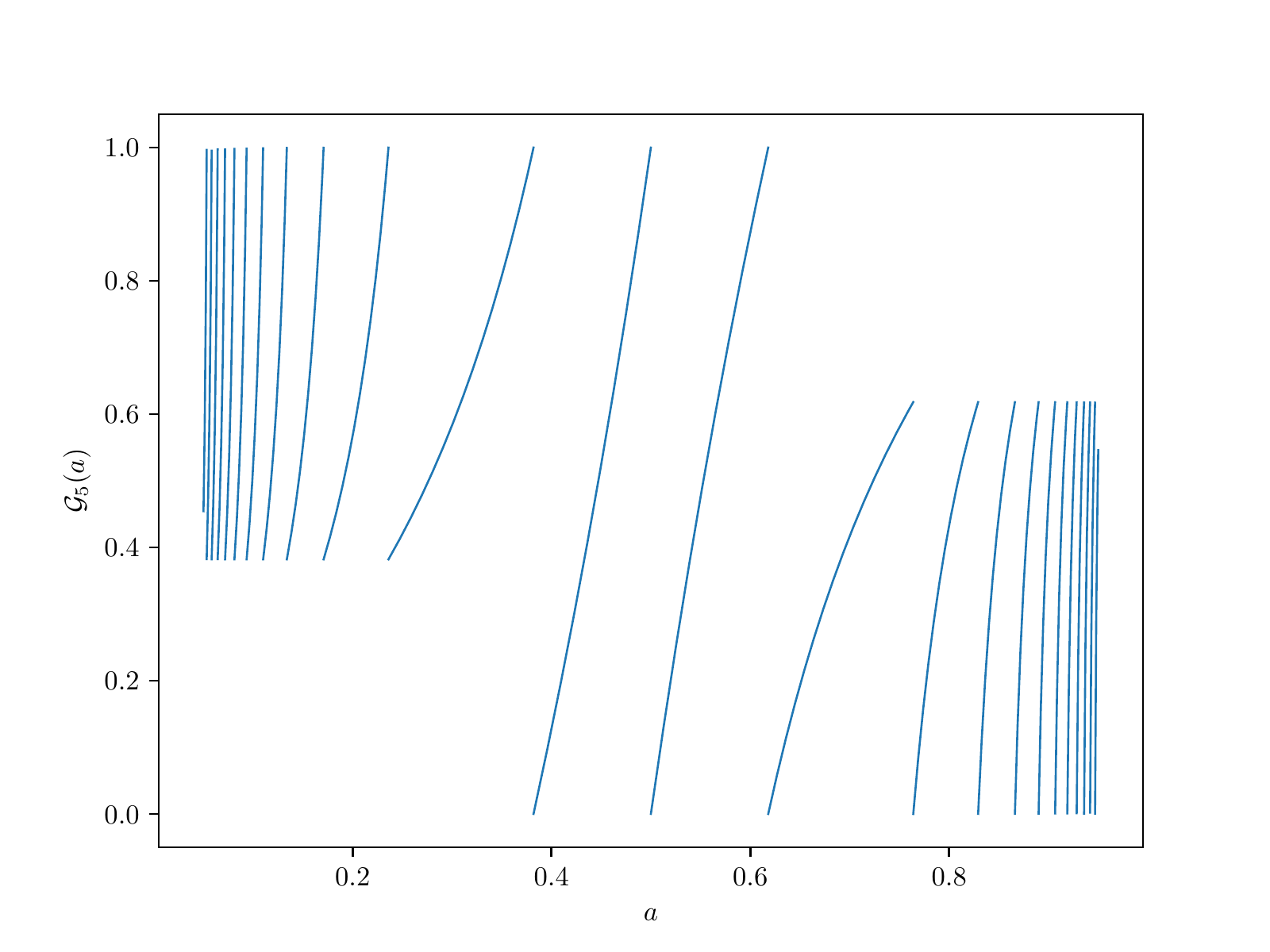}
\caption{The symmetric $G_5$-Gauss interval map $\mathscr{G}_5$.}
\end{figure}

\begin{theorem}
\label{theorem: natural extension and invariant measure for Gauss map}
Fix the notation as in \cref{remark: resolving ambiguity} and \cref{theorem: natural extension and invariant measure for Farey map}.
Define the maps $n_0^q, n_{q-2}^q : (0, 1] \to \NN$ for any $a \in (0, 1]$
by
\begin{equation}
n_i^q(a) = \min\{n \in \NN \mid \mathscr{F}_q^n(a) \not\in I_i^q\}
\end{equation}
when $i = 0, q - 2$, and consider the set $\mathsf{R}^q = \mathsf{S} \setminus \left((\mathsf{H}_0^q \cap \mathsf{V}_{q-2}^q) \cup (\mathsf{H}_{q-2}^q \cap \mathsf{V}_0^q)\right)$. The map $\widetilde{\mathscr{G}}_q : \mathsf{R}^q \to \mathsf{R}^q$ given for any $(a, s) \in \mathsf{R}^q$ by
\begin{equation}
\widetilde{\mathscr{G}}_q(a, s) :=
\begin{cases}
\widetilde{\mathscr{F}}_q^{n_0^q(a)}(a, s), & a \in I_0^q = (\lambda_q/(\lambda_q + 1), 1] \\
\widetilde{\mathscr{F}}_q(a, s), & a \notin I_0^q \cup I_{q-2}^q \\
\widetilde{\mathscr{F}}_q^{n_{q-2}^q(a)}(a, s), & a \in I_{q-2}^q = (0, 1/(\lambda_q+1])
\end{cases}
\end{equation}
is a.e. bijective, and preserves the Lebesgue measure $d\mu_{\widetilde{\mathscr{G}}_q} := dads$ on $\mathsf{R}^q$. Consequently, $\widetilde{\mathscr{G}}_q$ is a model of the natural extension of the map $\mathscr{G}_q : (0, 1] \to (0, 1]$ defined for every $a \in (0, 1]$ by
\begin{equation}
\mathscr{G}_q(a) :=
\begin{cases}
\mathscr{F}_q^{n_0^q(a)}(a), & a \in I_0^q \\
\mathscr{F}_q(a), & a \notin I_0^q \cup I_{q-2}^q \\
\mathscr{F}_q^{n_{q-2}^q(a)}(a), & a \in I_{q-1}^q
\end{cases}.
\end{equation}
Moreover, the map $\mathscr{G}_q$ preserves the finite measure
\begin{equation}
d\mu_{\mathscr{G}_q} :=
\begin{cases}
\left(R_{q, q - 2}^q(a, 1 - a) - R_{q, 0}^q(a, 1 - a)\right)da = \frac{\lambda_q\, da}{a(a + \lambda_q(1 - a))}, & a \in I_0^q = (\lambda_q/(\lambda_q+1), 1] \\
\frac{da}{a(1-a)}, & a \not\in I_0^q \cup I_{q-2}^q \\
\left(R_{q, q-1}^q(a, 1 - a) - R_{q, 1}^q(a, 1 - a)\right)da = \frac{\lambda_q\,da}{(1-a)((1-a) + \lambda_q a)}, & a \in I_{q-2}^q = (0, 1/(\lambda_q+1)]
\end{cases}
\end{equation}
on $(0, 1]$.
\end{theorem}

\begin{proof}
We first show that the map $\mathscr{F}_q : (0, 1] \to (0, 1]$ is intermittent, with exactly two indifferent fixed points at $a \to 0+0, 1-0$ by computing the derivative $\frac{d\mathscr{F}_q}{da}$. For any $i \in [q - 1]$, and any $a \in I_i^q = \left(\frac{x_{i+1}^q}{x_{i+1}^q + y_{i+1}^q}, \frac{x_i^q}{x_i^q + y_i^q}\right]$, we write $\varsigma_i = x_i^q + y_i^q$, and $\varsigma_{i+1}^q = x_{i+1}^q + y_{i+1}^q$, which gives
\begin{eqnarray*}
\frac{d\mathscr{F}_q}{da}(a) &=& \frac{\varsigma_{i+1}^q \left((\varsigma_{i+1}^q - \varsigma_i^q)a + (x_i^q - x_{i+1}^q)\right) - (\varsigma_{i+1}^q - \varsigma_i^q)\left(\varsigma_{i+1}^q a - x_{i+q}^q\right)}{\left((x_{i+1}^q - y_i^q)a + (x_i^q - x_{i+1}^q)\right)^2} \\
 &=& \frac{x_i^q \varsigma_{i+1}^q - \varsigma_i^q x_{i+1}^q}{\left((x_{i+1}^q - y_i^q)a + (x_i^q - x_{i+1}^q)\right)^2} \\
 &=& \frac{1}{\left((x_{i+1}^q - y_i^q)a + (x_i^q - x_{i+1}^q)\right)^2}
\end{eqnarray*}
where we use the fact that $\mathfrak{w}_i^q \wedge \mathfrak{w}_{i+1}^q = x_i^q y_{i+1}^q - x_{i+1}^q y_i^q = 1$. Similarly, we get
\begin{equation*}
\lim_{a \to \frac{x_{i+1}^q}{x_{i+1}^q + y_{i+1}^q}+0} \frac{d\mathscr{F}_q}{da}(a) = (x_{i+1}^q + y_{i+1}^q)^2,
\end{equation*}
and
\begin{equation*}
\lim_{a \to \frac{x_i^q}{x_i^q + y_i^q}-0} \frac{d\mathscr{F}_q}{da}(a) = (x_i^q + y_i^q)^2
\end{equation*}
at the end points of $I_i^q$. Since $(x_i + y_i^q)^2$ is equal to $1$ for $i = 0, q - 1$, and is strictly greater than $1$ for $i = 1, 2, \cdots, q - 2$, we have that $\mathscr{F}_q$ is expanding everywhere except at $a \to 0+0, 1-0$. It is immediately true that $\lim_{a \to 0+0} \mathscr{F}_q(a) = 0$ and $\lim_{a \to 1-0} \mathscr{F}_q(a) = 1$, proving that $0, 1$ are the only indifferent fixed points of $\mathscr{F}_q$.

We now accelerate the branches of $\mathscr{F}_q$ that correspond to $0$ and $1$. In what follows, denote the branches of $\mathscr{F}_q$ over $I_0^q$ and $I_{q-2}^q$ by $\mathscr{F}_{q,0}$ and $\mathscr{F}_{q,q-2}$. It can be easily seen that $n_0^q(a) = 1$ iff $a \in I_{0, 0}^q := \left(\frac{\lambda_q}{\lambda_q + 1}, \mathscr{F}_{q, 0}^{-1}\left(\frac{\lambda_q}{\lambda_q + 1}\right)\right]$, and by induction, we get for $n \geq 1$ that $n_0^q(a) = n$ iff $a \in I_{0, n}^q := \left(\mathscr{F}_{q, 0}^{-n+1}\left(\frac{\lambda_q}{\lambda_q + 1}\right), \mathscr{F}_{q, 0}^{-n}\left(\frac{\lambda_q}{\lambda_q + 1}\right)\right]$. Similarly, we for $n \geq 1$ have that $n_{q-2}^q(a) = n$ iff $a \in I_{q - 2, n}^q := \left(\mathscr{F}_q^{-n}\left(\frac{1}{\lambda_q + 1}\right), \mathscr{F}_{q, 0}^{-n+1}\left(\frac{1}{\lambda_q + 1}\right)\right]$. For $n \ge 1$, we define
\begin{equation*}
\mathsf{V}_{0,n}^q := \left\{(a, s) \in V_0^q, s < R_{q, q-2}(a, 1 - a)\right\},
\end{equation*}
and
\begin{equation*}
\mathsf{V}_{q-2,n}^q := \left\{(a, s) \in \mathsf{V}_{q-2}^q, s \geq R_{q, 1}(a, 1 - a)\right\}.
\end{equation*}
Note that $\{\mathsf{V}_{0,n}^q\}_{n=1}^\infty$ is a partition of $\mathsf{V}_0^q \setminus \mathsf{H}_{q-2}^q$, and that $\{\mathsf{V}_{q-2,n}^q\}_{n=1}^\infty$ is a partition of $\mathsf{V}_{q-2}^q \setminus \mathsf{H}_0^q$, and that $\mathsf{R}^q = \left(\cup_{n=1}^\infty \mathsf{V}_{0,n}^q\right) \cup \left(\cup_{i=1}^{q-3} \mathsf{V}_i^q\right) \cup \left(\cup_{n=1}^\infty \mathsf{V}_{q-2,n}^q\right)$. It is clear that the map $\widetilde{\mathscr{G}}_q$ when restricted to $\cup_{i=1}^{q-3} \mathsf{V}_i^q$ is equal to $\widetilde{\mathscr{F}}_q$, and that it sends $\cup_{i=1}^{q-3} \mathsf{V}_i^q$ bijectively to $\cup_{i=1}^{q-3} \mathsf{H}_i^q$. It thus remains to show that $\widetilde{\mathscr{G}}_q$ maps $\left(\cup_{n=1}^\infty \mathsf{V}_{0,n}^q\right) \cup \left(\cup_{n=1}^\infty \mathsf{V}_{q-2,n}^q\right)$ bijectively to $\left(\mathsf{H}_{q-2}^q \setminus \mathsf{V}_0^q\right) \cup \left(\mathsf{H}_0^q \setminus \mathsf{V}_{q-2}^q\right)$.

We prove that $\widetilde{\mathscr{G}}_q$ bijectively maps $\cup_{n=1} V_{q-2,n}^q$ to $\mathsf{H}_{0}^q \setminus \mathsf{V}_{q-2}^q$, and omit the (similar) proof that $\cup_{n=1}^\infty \mathsf{V}_{0,n}^q$ is bijectively mapped to $\mathsf{H}_{q-2}^q \setminus \mathsf{V}_0^q$. Towards this goal, we consider the curve
\begin{equation*}
\mathcal{U} := \left\{(a, s) \in \RR^2 \mid a \in (0, 1), s = \frac{1}{a(1-a)}\right\}.
\end{equation*}
Note that $\mathcal{U}$ is the upper boundary of $\mathsf{H}_{q-2}^q$ (and necessarily $\mathsf{S}$) in the plane. We define a sequence of curves $\left(\mathcal{U}_n^q\right)_{n=1}^\infty$ as follows: $\mathcal{U}_1^q : = \mathcal{U}$, and for any $n \geq 1$, $\mathcal{U}_{n+1}^q := \widetilde{\mathscr{F}}_q\left(\mathcal{U}_n^q \cap \{(a, s) \in \RR^2 \mid a \in I_{q-2}^q\}\right)$. That is, for every $n \geq 1$, we map the portion of $\mathcal{U}_n^q$ that lies in the $\mathsf{V}_{q-2}^q$ branch by $\widetilde{\mathscr{F}}_q$. By virtue of the definition of $\widetilde{\mathscr{F}}_q$ when it maps $\mathsf{V}_{q-2}^q$ to $\mathsf{H}_0^q$, we have that $\mathcal{U}_2^q$ agrees with the upper boundary of $\mathsf{H}_0^q$ in the plane, and that for every $n \geq 1$, the curve $\mathcal{U}_{n+1}^q$ lies \emph{below} $\mathcal{U}_n^q$ inside $\mathsf{H}_0^q$. For convenience, we write for any $(a, s) \in \mathsf{S}$ and $n \geq 1$ that $(a, s) \in \left[\mathcal{U}_{n+1}^q, \mathcal{U}_n^q\right)$ to indicate that $(a, s)$ lies between the curves $\mathcal{U}_{n+1}^q$ (inclusive), and $\mathcal{U}_n^q$ (exclusive). For any $n \geq 1$, we get from the definitions of $\mathsf{V}_{q-2, n}^q$, $I_{q-2, n}^q$, $\mathcal{U}_n^q$ and $\mathcal{U}_{n+1}^q$ that
\begin{equation*}
\widetilde{\mathscr{F}}_q^k \left(\mathsf{V}_{q-2, n}^q\right) = \left\{(a, s) \in \mathsf{S} \mid a \in I_{q-2, n - k}^q, (a, s) \in \left[\mathcal{U}_{k+2}^q, \mathcal{U}_{k+1}^q\right)\right\}
\end{equation*}
for $k = 0, 1, \cdots, n - 1$, and that
\begin{equation*}
\widetilde{\mathscr{F}}_q^n\left(\mathsf{V}_{q-2,n}^q\right) = \left\{(a, s) \in \mathsf{S} \mid a \in \cup_{i=0}^{q-3} I_i^q, (a, s) \in \left[\mathcal{U}_{n+2}^q, \mathcal{U}_{n+1}^q\right)\right\}.
\end{equation*}
That is, for each $n \geq 1$, $\widetilde{\mathscr{G}}_q$ bijectively map $\mathsf{V}_{q-2,n}^q$ to $\widetilde{\mathscr{F}}_q^n\left(\mathsf{V}_{q-2, n}^q\right) \subset \mathsf{H}_0^q \setminus \mathsf{V}_{q-2}^q$, with the left and right boundaries of $\widetilde{\mathscr{F}}_q^n\left(\mathsf{V}_{q-2, n}^q\right)$ agreeing with those of $\mathsf{H}_0^q \setminus \mathsf{V}_{q-2}^q$, and the bottom boundary of $\widetilde{\mathscr{F}}_q^n\left(\mathsf{V}_{q-2, n}^q\right)$ agreeing with the top boundary of $\widetilde{\mathscr{F}}_q^{n+1}\left(\mathsf{V}_{q-2, n+1}^q\right)$. The function $\widetilde{\mathscr{G}}_q$ is a composition of $\widetilde{\mathscr{F}}_q$, and so the mapping in question preserves the measure $d\mu_{\widetilde{F}_q}$.

Finally, integrating the measure $d_{\widetilde{\mathscr{G}}_q}$ over the $s$ fibers gives the measure $d_{\mathscr{G}_q}$ in the statement of the theorem.
\end{proof}

\begin{figure}
\centering
\includegraphics[width=.8\textwidth]{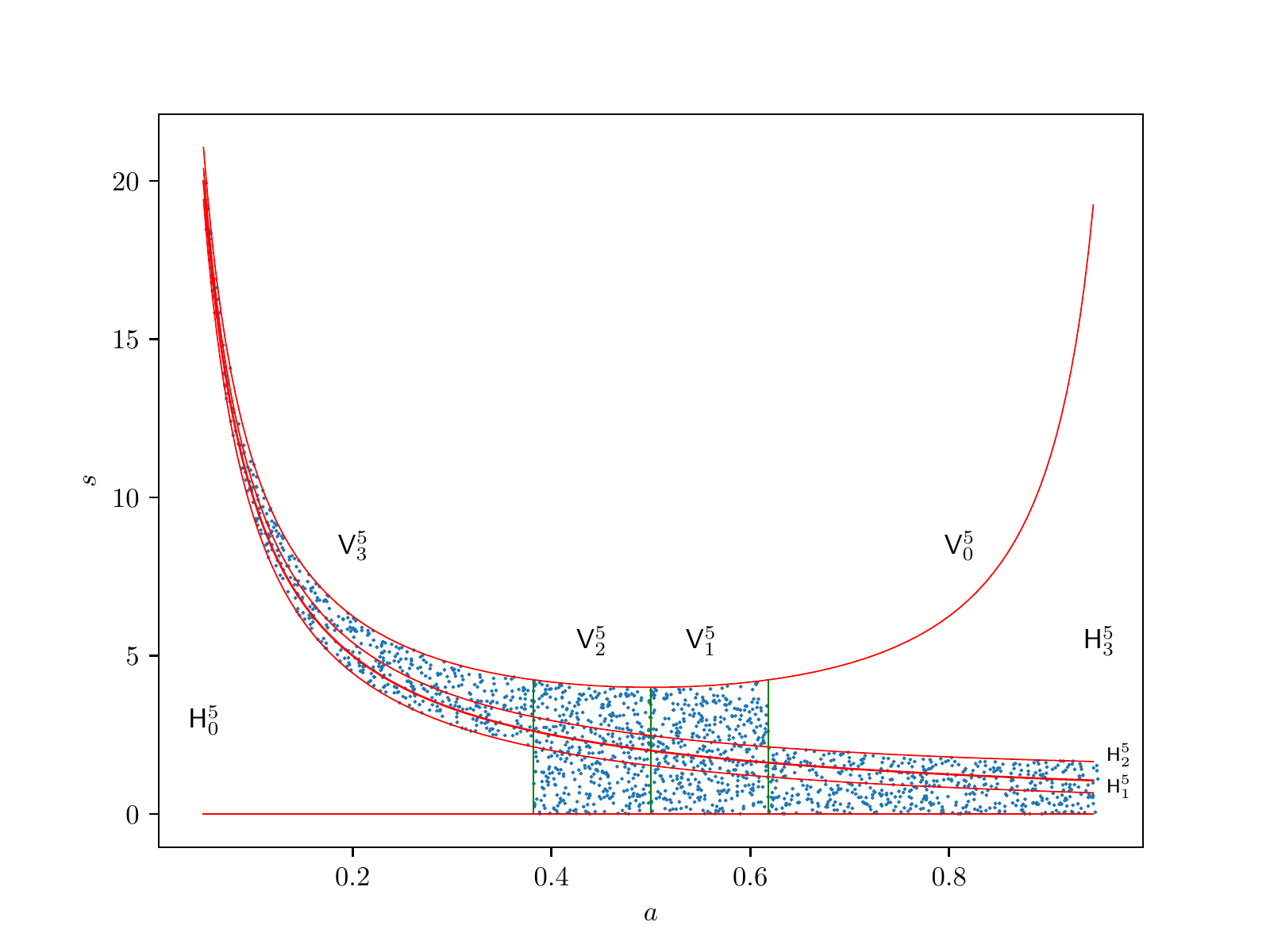}
\caption{An orbit of the natural extension $\widetilde{\mathscr{G}}_5$ with the regions $\{\mathsf{V}_i^3\}_{i=0}^3$ and $\{\mathsf{H}_i^3\}_{i=0}^3$ from \cref{proposition: side identification} indicated.}
\bigbreak
\includegraphics[width=.8\textwidth]{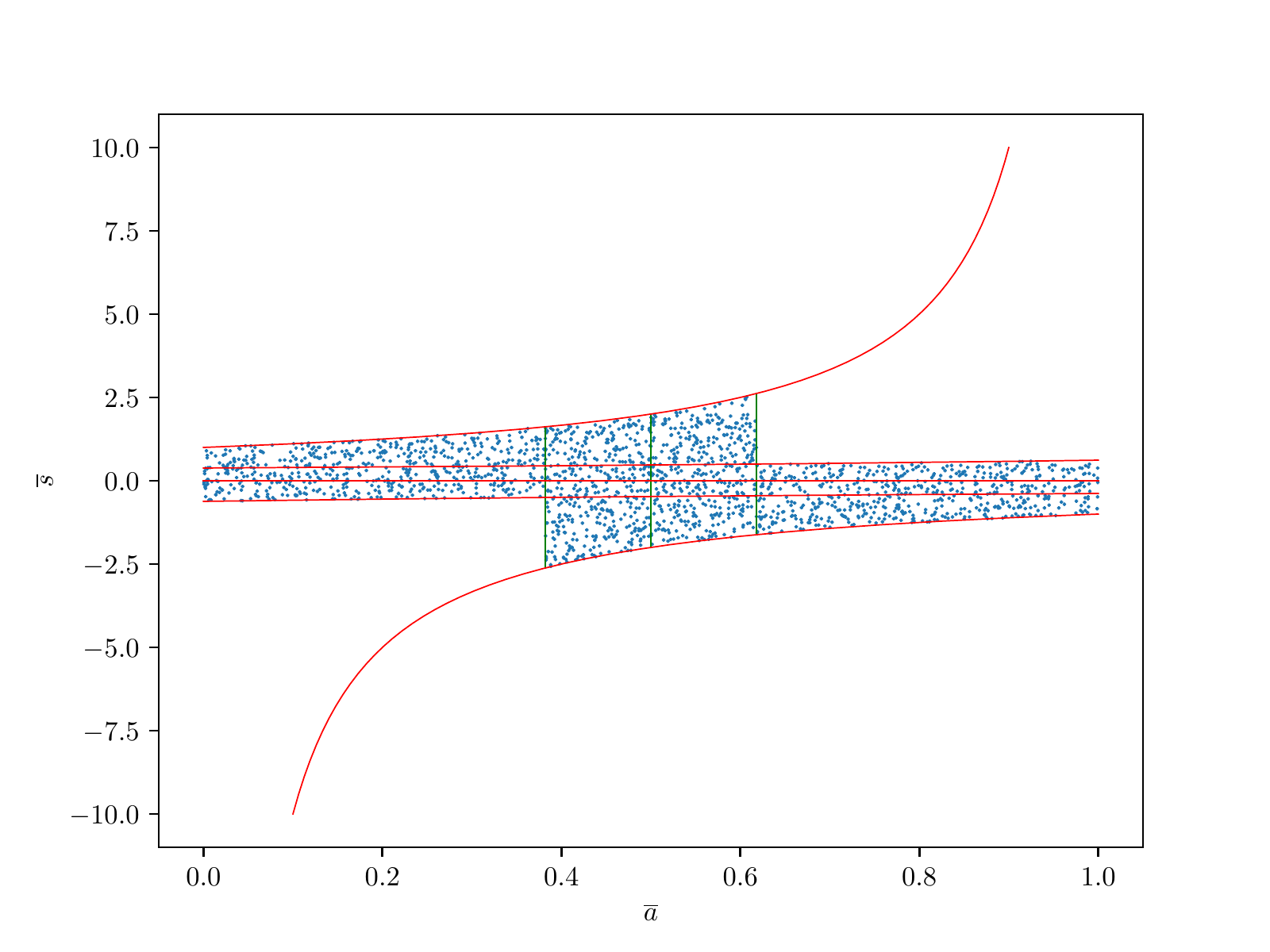}
\caption{An orbit of the natural extension $\widetilde{\mathscr{G}}_5$ conjugated by $(a, s) \mapsto (\overline{a}, \overline{s}) := (a, s - 1/a)$.}
\end{figure}


\clearpage

\begin{thebibliography}{9}

\bibitem{Arnoux1994-mk} Arnoux, P. (1994). Le codage du flot g\'{e}od\'{e}sique sur la surface modulaire. Enseignement des Mathematiques, 40, 29–48.

\bibitem{Arnoux2000-bt} Arnoux, P., \& Hubert, P. (2000). Fractions continues sur les surfaces de Veech. Journal d'Analyse Math\'{e}matique, 81(1), 35–64.

\bibitem{Arnoux2018-dg} Arnoux, P., \& Labb\'{e}, S. (2018). On some symmetric multidimensional continued fraction algorithms. Ergodic Theory and Dynamical Systems. doi:10.1017/etds.2016.112

\bibitem{Arnoux1993-xh} Arnoux, P., \& Nogueira, A. (1993). Mesures de Gauss pour des algorithmes de fractions continues multidimensionnelles. In Annales scientifiques de l'Ecole normale sup\'{e}rieure (Vol. 26, pp. 645–664).

\bibitem{Arnoux2017-lm} Arnoux, P., \& Schmidt, T. A. (2017). Natural extensions and Gauss measures for piecewise homographic continued fractions. arXiv [math.DS]. http://arxiv.org/abs/1709.05580

\bibitem{Davis2018-al} Davis, D., \& Lelievre, S. (2018, October 26). Periodic paths on the pentagon, double pentagon and golden L. arXiv [math.DS]. http://arxiv.org/abs/1810.11310

\bibitem{Janvresse2011-fl} Janvresse, E., Rittaud, B., \& De La Rue, T. (2011). Dynamics of $\lambda$-continued fractions and $\beta$-shifts. arXiv [math.PR]. http://arxiv.org/abs/1103.6181

\bibitem{Lang2016-qs} Lang, C. L., \& Lang, M. L. (2016). Arithmetic and geometry of the Hecke groups. Journal of Algebra, 460, 392–417.

\bibitem{Rauzy1979-tf} Rauzy, G. (1979). \'{E}changes d'intervalles et transformations induites. Acta Arithmetica, 34(4), 315–328. Accessed 25 October 2016

\bibitem{Rosen1954-gp} Rosen, D. (1954). A class of continued fractions associated with certain properly discontinuous groups. Duke Mathematical Journal, 21(3), 549–563.

\bibitem{Smillie2010-vf} Smillie, J., \& Ulcigrai, C. (2010). Geodesic flow on the Teichm\"{u}ller disk of the regular octagon, cutting sequences and octagon continued fractions maps. Dynamical numbers--interplay between dynamical systems and number theory, Contemp. Math, 29–65.

\bibitem{Smillie2011-ql} Smillie, J., \& Ulcigrai, C. (2011). Beyond Sturmian sequences: coding linear trajectories in the regular octagon. Proceedings of the London Mathematical Society. Third Series, 102(2), 291–340.

\bibitem{Veech1982-vh} Veech, W. A. (1982). Gauss Measures for Transformations on the Space of Interval Exchange Maps. Annals of mathematics, 115(2), 201–242.

\bibitem{dia} Taha, D. (2019). The Boca-Cobeli-Zaharescu Map Analogue for the Hecke Triangle Groups $G_q$. arXiv [math.DS]. https://arxiv.org/abs/1810.10668

\bibitem{Zorich1996-hj} Zorich, A. (1996). Finite Gauss measure on the space of interval exchange transformations. Lyapunov exponents. In Annales de l'institut Fourier (Vol. 46, pp. 325–370).

\end{thebibliography}
\end{document}